\documentclass[11pt,letterpaper]{amsart}
\usepackage{amsmath,amssymb,xy,array}
\usepackage[T1]{fontenc}
\usepackage{amsfonts}
\usepackage{amsmath}
\usepackage{amssymb}
\usepackage{amsthm}
\usepackage{hyperref}
\usepackage{graphicx}
\usepackage{fancyhdr}
\usepackage{tikz}
\usetikzlibrary{cd}
\usepackage{longtable}
\usepackage{textcomp}
\usepackage[utf8]{inputenc}
\usetikzlibrary{trees}

\usepackage{amsfonts}
\usepackage{amssymb,amsthm}
\usepackage{verbatim}
\usepackage{hyperref} 
\usepackage[T1]{fontenc}
\usepackage{amsmath}
\usepackage{hyperref}
\usepackage{enumerate}
\usepackage{tikz}
\usepackage{color}

\newcommand{\G}{\mathbb{G}}

\newcommand{\C}{\mathbb C}
\newcommand{\p}{\mathbb P}

\newcommand{\F}{\mathbb {F}}
\DeclareMathOperator{\codim}{codim}
\DeclareMathOperator{\rk}{rank}

\DeclareMathOperator{\Sec}{Sec}

\DeclareMathOperator{\Aff}{Aff} 

\newcommand{\QED}{\ifhmode\unskip\nobreak\fi\quad {\rm Q.E.D.}} 

\newcommand\Span[1]{\langle{#1}\rangle}

\renewcommand{\sec}{\mathbb{S}ec}
\DeclareMathOperator{\expdim}{expdim}

\newtheorem{thm}{Theorem}[section]
\newtheorem{Lemma}[thm]{Lemma}

\newtheorem{Proposition}[thm]{Proposition}
\newtheorem{Corollary}[thm]{Corollary}
\newtheorem{Conjecture}[thm]{Conjecture}
\newtheorem*{thm*}{Theorem}

\theoremstyle{definition}
\newtheorem{Notation}[thm]{Notation}

\newtheorem{Definition}[thm]{Definition}
\newtheorem{Remark}[thm]{Remark}

\begin{document}
\title[\resizebox{4.5in}{!}{Bronowski's conjecture and the identifiability of projective varieties}]{Bronowski's conjecture and the identifiability of projective varieties}

\author[Alex Massarenti]{Alex Massarenti}
\address{\sc Alex Massarenti\\ Dipartimento di Matematica e Informatica, Universit\`a di Ferrara, Via Machiavelli 30, 44121 Ferrara, Italy}
\email{msslxa@unife.it}

\author[Massimiliano Mella]{Massimiliano Mella}
\address{\sc Massimiliano Mella\\ Dipartimento di Matematica e Informatica, Universit\`a di Ferrara, Via Machiavelli 30, 44121 Ferrara, Italy}
\email{mll@unife.it}

\date{\today}

\maketitle

\begin{abstract}
Let $X\subset\mathbb{P}^{hn+h-1}$ be an irreducible and non-degenerate
variety of dimension $n$. The Bronowski's conjecture predicts that $X$
is $h$-identifiable if and only if the general $(h-1)$-tangential
projection $\tau_{h-1}^X:X\dasharrow\mathbb{P}^n$ is birational. In
this paper we provide counterexamples to this conjecture. Building on
the ideas that led to the counterexamples we manage to prove an
amended version of the Bronowski's conjecture for a wide class of
varieties and to reduce the identifiability problem for projective
varieties to their secant defectiveness.
\end{abstract}

\setcounter{tocdepth}{1}
\tableofcontents

\section{Introduction}

Let $X\subset\mathbb{P}^{N}$ be an irreducible and non-degenerate
variety of dimension $n$, and $p\in\left\langle
  x_1,\dots,x_h\right\rangle$ a general point in the linear span of
$h$ general points $x_1,\dots,x_h\in X$. The variety $X$ is
$h$-identifiable if $\left\langle x_1,\dots,x_h\right\rangle$ is the
only $(h-1)$-plane spanned by $h$ points of $X$ that passes through
$p$, when $N = hn+h-1$ we say that $X$ is generically identifiable. It is easy to guess that generic identifiability is a rare property that deserves attention.

Identifiability is interesting both from the viewpoint of birational
geometry, providing unexpected birational maps onto the
projective space, and in many applications especially when the ambient
projective space parametrizes tensors. 

We will denote by $\tau^X_{h-1}:X\dasharrow
X_{h-1}\subset\mathbb{P}^{N_{h-1}}$ the projection from the span of
the tangent spaces of $X$ at $h-1$ general points, and we will refer
to it as a general $(h-1)$-tangential projection. The pioneering work
of J. Bronowski \cite{Bro32} was meant to relate the
$h$-identifiability of $X$ with the birationality of
$\tau^X_{h-1}$. His main legacy is the so called Bronowski's
conjecture, see \cite{CR06} for a modern statement.  

\begin{Conjecture}\label{CBroC}
Let $X\subset\mathbb{P}^{hn+h-1}$ be an irreducible and non-degenerate
variety of dimension $n$. The variety $X$ is $h$-identifiable if and
only if a general $(h-1)$-tangential projection $\tau_{h-1}^X:X\dasharrow\mathbb{P}^n$ is birational.   
\end{Conjecture}

The state of the art on Conjecture~\ref{CBroC} is the
following: it is known that $h$-identifiability implies the birationality of a general $\tau_{h-1}^X$, see \cite{CMR04} in a special case and \cite[Corollary 4.5]{CR06} in general. The full conjecture is true for curves \cite{CK22}, smooth surfaces \cite{CR06} and linearly normal smooth $3$-folds when $h = 2$ \cite{CMR04}.

In modern times the study of identifiability has been greatly
generalized with the help of secant varieties and tools developed for
the study of various types of contact loci,  \cite{CC02}  \cite{CR06}
\cite{Me06} \cite{BBC} \cite{CM19}\cite{CM21}.

The $h$-secant variety
$\sec_h(X)\subset\mathbb{P}^N$ of $X$ is the Zariski closure of the
union of the $(h-1)$-planes spanned by collections of $h$ points of
$X$.
The expected dimension of $\mathbb{S}ec_{h}(X)$ is
$$\expdim(\mathbb{S}ec_{h}(X)):= \min\{nh+h-1,N\}.$$
The actual
dimension of $\mathbb{S}ec_{h}(X)$ may be  smaller than the expected
one. The variety $X$ is $h$-defective if
$\dim(\mathbb{S}ec_{h}(X)) 
< \expdim(\mathbb{S}ec_{h}(X))$ and 
$h$-identifiable if through a  general point of $\sec_h(X)$ there passes a unique
$(h-1)$-plane spanned by $h$ points of $X$.

In particular, in \cite{CR06} the authors  proposed a modified version of the Bronowski's conjecture that apparently generalizes it. 

\begin{Conjecture}\cite[Remark 4.6.]{CR06}\label{CBro}
Let $X\subset\mathbb{P}^N$ be an irreducible and non-degenerate
variety. If a general $(h-1)$-tangential projection of $X$ is
birational and its image is a variety of minimal degree then $X$ is
$h$-identifiable and $\sec_h(X)$ has degree $\binom{h + c_h}{h}$, where $c_h = \codim_{\mathbb{P}^N}(\sec_h(X))$.   
\end{Conjecture}
Note that by \cite[Theorem 4.2]{CR06} $\binom{h + c_h}{h}$ is the smallest possible degree for $\sec_h(X)$, and if you are puzzled by the addition of the minimal degree requirement observe that in the classical Bronowsky's conjecture the corresponding varieties are projective spaces and hence minimal degree varieties. Recently, in \cite[Corollary 3.9]{CK22} J. Choe and S. Kwak observed that
Conjecture \ref{CBro} is implied by Conjecture~\ref{CBroC}.

In Theorem \ref{mainCE} we give counterexamples to both Conjecture~\ref{CBro} and Conjecture~\ref{CBroC}. As a sample we have the following result:

\begin{thm}\label{mainA}
Let $\Gamma\subset\mathbb{P}^N$ be a rational normal curve of degree $N\geq 7$ and assume $h = \frac{N+1}{2r}$ to be an integer. Then $\sec_r(\Gamma)$ gives a counterexample to Conjecture~\ref{CBroC} and Conjecture~\ref{CBro}.  
\end{thm}

Note that the variety  considered in Theorem \ref{mainA} has
degenerate Gauss map or, in a different dictionary, it is $1$-tangentially weakly defective. Despite various attempts we have not been able to extend the counterexamples to a wider class of varieties. It is therefore tempting to state the following:

\begin{Conjecture}\label{con:Bro_mio}
Let $X\subset\mathbb{P}^{N}$ be an irreducible and non-degenerate variety with non degenerate Gauss map. The variety $X$ is $h$-identifiable if and only if a general $(h-1)$-tangential projection $\tau_{h-1}^X:X\dasharrow X_{h-1}\subset\mathbb{P}^{N_{h-1}}$ is birational. 
\end{Conjecture}

As already observed, thanks to \cite[Corollary 4.5]{CR06} if $X$ is $h$-identifiable then $\tau_{h-1}^X$ is birational, and hence only one direction of the Conjecture~\ref{con:Bro_mio} is open.

This new formulation of the conjecture echoes results obtained in \cite{CM19} that connected non identifiability and defectiveness. In our second main result we improve the core result of \cite{CM19} and give the ultimate relation between defectiveness and identifiablity.

\begin{thm}\label{mainB}
Let $X\subset\mathbb{P}^N$ be an irreducible and non-degenerate variety of dimension $n$, $h\geq 1$ an integer, and assume that 
\begin{itemize}
\item[-] $(h+1)n+h \leq N$;
\item[-] $X$ has non degenerate Gauss map;
\item[-] $X$ is not $(h+1)$-defective. 
\end{itemize}
Then $X$ is $h$-identifiable. 
\end{thm}

As for the Bronowski's conjecture, the assumption on the non
degeneracy of the Gauss map of $X$ can not be dropped, see Remark
\ref{n1twdn} for an explicit example. The inequality $(h+1)n+h \leq N$
is also necessary. For instance, the degree six Veronese embedding
$V^2_6\subset\mathbb{P}^{27}$ of $\mathbb{P}^2$ is not $10$-defective \cite{AH95}
but through a general point of $\sec_9(V^2_6)$ there are two
$8$-planes intersecting $V^2_6$ in nine points \cite[Theorem 1.1]{COV17}.  

As a straightforward consequence of Theorem \ref{mainB} we obtain the following special case of Conjecture~\ref{con:Bro_mio}.

\begin{Corollary}\label{cor:Bro_mio_vera} 
Let $X\subset\mathbb{P}^{N}$ be an irreducible and non-degenerate variety of dimension $n$. Assume that $(h+1)n+h \leq N$ and $X$ is not $(h+1)$-defective. Then
Conjecture~\ref{con:Bro_mio} holds true for $X$ and $h$. 
\end{Corollary}

The above result should be confronted with \cite[Theorem 1.5]{CK22} where the same statement is proved for Conjecture~\ref{CBro}, see also Remark~\ref{rem:Th15} for further details.

Theorem \ref{mainB} allows  to translate all known bounds for non secant
defectiveness into bounds for identifiability.   
As a sample we focus here on  Segre-Veronese varieties.

Let
$\pmb{n}=(n_1,\dots,n_r)$ and $\pmb{d} = (d_1,\dots,d_r)$ be two
$r$-uples of positive integers, with $n_1\leq \dots \leq n_r$ and
$N(\pmb{n},\pmb{d})=\prod_{i=1}^r\binom{n_i+d_i}{n_i}-1$. The
Segre-Veronese variety $SV_{\pmb{d}}^{\pmb{n}}$ is the image in
$N(\pmb{n},\pmb{d})$ of
$\mathbb{P}^{n_1}\times\dots\times\mathbb{P}^{n_r}$ via the embedding
induced by $ 
\mathcal{O}_{\mathbb{P}^{\pmb{n}} }(d_1,\dots,
d_r)=\mathcal{O}_{\mathbb{P}(V_1^{*})}(d_1)\boxtimes\dots\boxtimes
\mathcal{O}_{\mathbb{P}(V_1^{*})}(d_r)$. As a direct consequence of \cite[Theorem 1.1]{LMR20} and Theorem \ref{mainB} we have the following: 

\begin{thm}\label{mainC}
Let $SV_{\pmb{d}}^{\pmb{n}}\subset \mathbb{P}^{N(\pmb{n},\pmb{d})}$ be a Segre-Veronese variety. If 
$$h < \left\lfloor\frac{d_j}{n_j+d_j}\frac{1}{\sum_{i=1}^r n_i + 1}\prod_{i=1}^r \binom{n_i+d_i}{d_i}\right\rfloor,$$
where $\frac{n_j}{d_j} = \max_{1\leq i\leq r}\left\lbrace\frac{n_i}{d_i}\right\rbrace$, then $SV_{\pmb{d}}^{\pmb{n}}$ is $h$-identifiable.
\end{thm}

The expected generic rank for Segre-Veronese varieties is given by a
polynomial of degree $\sum_i n_i$ in the $d_i$ and of degree $\sum_i
d_i - 1$ in the $n_i$. Theorem \ref{mainC} gives a polynomial bound of
degree $\sum_i n_i$ in the $d_i$, while in the $n_i$ we have a
polynomial bound of degree $\sum_id_i-2$. At the best of our knowledge
so far this is the best general bound, without any technical
assumptions on the involved parameters, for identifiability of
Segre-Veronese varieties. Pushing a bit further the result in
\cite{LP13}  we are able to
match the best possible bound for products of lines, extending to the
Segre-Veronese setting \cite[Theorem 26]{CM19}:

\begin{thm}\label{SVP1}
Let $SV_{\pmb{d}}^{(1,\dots,1)}\subset \mathbb{P}^{N(\pmb{n},\pmb{d})}$ be a Segre-Veronese variety with $\pmb{n} = (1,\dots,1)$. If 
$$h < \left\lfloor\frac{\prod_{i=1}^r(d_i+1)}{r+1}\right\rfloor$$
then $SV_{\pmb{d}}^{(1,\dots,1)}$ is $h$-identifiable.
In particular when $\frac{\prod_{i=1}^r(d_i+1)}{r+1}$ is an integer and $d_i\geq 3$, the variety $SV_{\pmb{d}}^{(1,\dots,1)}$ is $h$-identifiable for all subgeneric ranks and it is not generically identifiable.
\end{thm}
We stress that Theorem \ref{SVP1} gives the generic identifiability of all
sub-generic binary mixed tensors, qubits in the quantum information
dictionary,  in the perfect case, that is when
$\frac{\prod_{i=1}^r(d_i+1)}{r+1}$ is an integer, and all but the last
one in general.

The analysis in Theorem \ref{SVP1} allowed us to find out a new not
identifiable Segre-Veronese variety.
The variety
$SV_{(2,2,2)}^{(1,1,1)}\subset \mathbb{P}^{26}$ is $h$-identifiable
for $h\leq 5$, by Theorem~\ref{SVP1}. In Remark \ref{SVn6id} we
show that it is not $6$-identifiable since through a general point of
$\sec_6(SV_{(2,2,2)}^{(1,1,1)})$ there are exactly two $5$-planes
intersecting $SV_{(2,2,2)}^{(1,1,1)}$ in six points.  

\subsection*{Organization of the paper}
All through the paper we work over an algebraically closed field $k$
of characteristic zero. In Section \ref{sec2} we prove Theorem
\ref{mainA}. In Section \ref{sec3} we plug the Bronowski's conjecture
into the relation between secant defectiveness and identifiability and
prove Theorem \ref{mainB}. Finally, in Section \ref{sec4} we apply
Theorem \ref{mainC} to derive bounds for identifiability from known
bounds for non secant defectiveness.

\section{Counterexamples to the Bronowski's conjecture}\label{sec2}

Let $X\subset\mathbb{P}^N$ be an irreducible and non-degenerate 
variety of dimension $n$ and let $\Gamma_h(X)\subset X\times \dots
\times X\times\G(h-1,N)$, where $h\leq N-n+1$, be the closure of the graph
of the rational map $\alpha: X\times\dots\times X \dasharrow
\G(h-1,N)$ taking $h$ general points to their linear span. Observe that since $h\leq N-n+1$ an $(h-1)$-plane spanned by $h$ general points of $X$ intersect $X$ in a finite number of points. So $\alpha$ is generically finite and $\Gamma_h(X)$ is irreducible
and reduced of dimension $hn$. Let $\pi_2:\Gamma_h(X)\to\G(h-1,N)$ be
the natural projection, and
$\mathcal{S}_h(X):=\pi_2(\Gamma_h(X))\subset\G(h-1,N)$. Again
$\mathcal{S}_h(X)$ is irreducible and reduced of dimension
$hn$. Finally, consider
$$\mathcal{I}_h=\{(x,\Lambda) \: | \: x\in \Lambda\}\subset\mathbb{P}^N\times\G(h-1,N)$$
with the projections $\pi_h^X$ and $\psi_h^X$ onto the factors. The abstract $h$-secant variety is the irreducible variety
$$\Sec_{h}(X):=(\psi_h^X)^{-1}(\mathcal{S}_h(X))\subset \mathcal{I}_h.$$
The $h$-secant variety is defined as
$$\sec_{h}(X):=\pi_h^X(\Sec_{h}(X))\subset\mathbb{P}^N.$$
It immediately follows that $\Sec_{h}(X)$ is an $(hn+h-1)$-dimensional
variety with a $\mathbb{P}^{h-1}$-bundle structure over the open subset of
$\mathcal{S}_h(X)$ parametrizing independent points. We say that $X$ is $h$-defective if
$\dim\sec_{h}(X)<\min\{\dim\Sec_{h}(X),N\}$.  

\begin{Definition}
When $\pi_h^X:\Sec_h(X)\rightarrow\sec_h(X)$ is generically finite (note that this is equivalent to say that $X$ is not $h$-defective) we will call
its degree the $h$-secant degree of $X$, and we will say that $X$ is
$h$-identifiable when its $h$-secant degree is one.  
\end{Definition}

In the study of identifiability a central notion is that of tangential weak defectiveness.
\begin{Definition}\cite{BBC}
Let $x_1,\dots,x_h\in X$ be general points. The $h$-tangential contact locus $\Gamma_{x_1,\dots,x_h}$ of $X$ with respect to $x_1,\dots,x_h$ is the closure in $X$ of
the union of all the irreducible components which contain at least one
of the $x_i$ of the locus of points of $X$ where $\left\langle
  T_{x_1}X,\dots, T_{x_h}X\right\rangle$ is tangent to $X$. Let
$\gamma_{x_1,\dots,x_h}$ be the largest dimension of the components of
$\Gamma_{x_1,\dots,x_h}$. If $\gamma_{x_1,\dots,x_h} > 0$ we say that
$X$ is $h$-tangentially weakly defective.
\end{Definition}
  
\begin{Remark} Note that a variety $X$ is 1-tangentially weakly defective
  if and only if its Gauss map is degenerate. 
By \cite[Proposition 14]{CM19}, see also  \cite[Proposition 2.4]{CO12}
for a similar result, if $X$ is not $h$-tangentially weakly defective
then it is $h$-identifiable.
It is important to stress that the two notions are not equivalent.  The Grassmannian
$\mathbb{G}(2,7)$ parametrizing planes in $\mathbb{P}^7$ is
$3$-tangentially weakly defective but it is $3$-identifiable
\cite[Proposition 1.7]{BV18}, see also Remark~\ref{rem:curv_no_bro}.
\end{Remark}

\begin{Notation}
Let $X\subset\p^N$ be an irreducible and non-degenerate variety. A general $h$-tangential projection of $X$ is a linear projection $\tau_{x_1,\dots,x_h}^X:X\dasharrow X_{h}\subset\mathbb{P}^{N_{h}}$ from the linear span of $h$ tangent
spaces $\left\langle T_{x_1}X,\dots, T_{x_h}X\right\rangle$ where
$x_1,\dots,x_h\in X$ are general points.

When there will be no danger of confusion we will denote a general $h$-tangential projection $\tau_{x_1,\dots,x_h}^X$ simply by $\tau_{h}^X$, a general $h$-tangential contact locus $\Gamma_{x_1,\dots,x_h}$ by $\Gamma_{h}$ and its dimension by $\gamma_h$. 
\end{Notation}

\begin{Remark}
It is easy to see that if $nh+h-1\leq N$ then $X$ is $h$-defective if and only if a general $(h-1)$-tangential projection of $X$ has fibers of positive dimension \cite[Proposition 3.5]{CC02}, see
Lemma~\ref{taupi} for further relations.
\end{Remark}

\begin{Definition}
Let $X\subset\mathbb{P}^N$ be an irreducible, non-degenerate and non $1$-tangentially weakly defective
variety. We say that $X$ is an $h$-Bronowski's variety if it satisfies Conjecture~\ref{con:Bro_mio} for the integer $h$ that is if either $\tau^X_h$ is not birational or $X$ is $h$-identifiable.
\end{Definition}

Note that, by \cite[Corollary 4.5]{CR06}, if a variety is $h$-identifiable then it is $h$-Bronowski. In \cite{CK22} the authors proved Conjecture~\ref{CBro} for curves. Here we propose a slightly modified proof, that led us to built the counterexamples described in Theorem~\ref{mainA}. 

All through the paper when we say that a map is of fiber type we mean that the general fiber of the map has positive dimension. 

\begin{Proposition}\label{Bro_cur}\cite[Theorem 1.5]{CK22}
Let $X\subset\mathbb{P}^N$ be an irreducible and non-degenerate
curve. If $N\geq 3$ then $X$ is an $h$-Bronowski's variety for $h\leq
\lfloor\frac{N}{2}\rfloor +1$.  
\end{Proposition}
\begin{proof}
Set $d = \deg(X)$. Since $X$ is non-degenerate we have $d\geq N$. If
$d = N$ then $X$ is smooth and hence it is a rational normal curve. In
this case $X$ is $h$-identifiable and hence it is $h$-Bronowski.

Assume now that $d > N$ and $X$ is not $h$-identifiable. If $\sec_h(X) \neq \mathbb{P}^N$ let $p\in\sec_h(X)$ be a general
point. Then, by the Terracini's lemma \cite{Te11} $X\subset
T_p\sec_h(X)$ is degenerate. Therefore, we may assume that $\sec_h(X)
= \mathbb{P}^N$. Furthermore, if $2h-1 > N$ then a general
$(h-1)$-tangential projection is of fiber type so we may assume that
$2h-1 = N$. Let $x_1,\dots,x_{h-1}\in X$ be general points and $p$ a
general point of the span $\left\langle
  x_1,\dots,x_{h-1}\right\rangle$. If $T_p\sec_{h-1}(X)$ has dimension
smaller that $2h-3$ then by the Terracini's lemma \cite{Te11} $X$ is
$(h-1)$-defective and hence, again by the Terracini's lemma
\cite{Te11}, $X\subset T_p\sec_{h-1}(X)$ is degenerate. So we may
assume that $T_p\sec_{h-1}(X)$ has dimension $2h-3$.  

Assume that $T_p\sec_{h-1}(X) = \left\langle T_{x_1}X,\dots,T_{x_{h-1}}X\right\rangle$ intersects $X$ in at most $d-2$ points counted with multiplicity. Then a general hyperplane containing
$T_p\sec_{h-1}(X)$ intersects $X$ in at least $d-(d-2)$ additional
points, and hence $\tau_{x_1,\dots,x_{h-1}}^X$ is not birational.        

Therefore, to conclude it is enough to prove that $T_p\sec_{h-1}(X)$
intersects $X$ in at most $d-2$ points counted with multiplicity. Up
to projecting from $\left\langle T_{x_1}X,\dots,T_{x_{h-2}}X\right\rangle$ we may assume that
$X\subset\mathbb{P}^3$ and $h = 2$. Assume that a general tangent line
$T_xX$ of $X$ intersects $X$ in at least $d-1$ points counted with
multiplicity and assume the existence of surface $Y\subset\mathbb{P}^3$ of degree
$d-2$ containing $X$. Then $T_{x}X\subset Y$ for all $x\in X$ and so
$Y$ contains the tangential variety of $X$. Therefore, either there is no surface of degree $d-2$ containing $X$ or such a surface 
must have the tangential variety of $X$ as an irreducible component. In particular, $X$ can not be
cut out by surfaces of degree $d-2$. So, in the terminology of
\cite{GLP83}, $X$ is not $C_{d-2}$ and by \cite[Corollary in the
Introduction]{GLP83} $X$ is smooth. Finally, since $X$ is
non-degenerate \cite[Theorem 3.1]{Kaj86} yields that $T_xX$ intersects
$X$ only at $x$. Since we are in the case $d > N$ we have $d>3$ and this is enough to conclude. 
\end{proof}

\begin{Remark}
Let $X$ be a variety with $h$-secant degree $d_h > 1$. Then $X$ is $h$-Bronowski if and only if a general $(h-1)$-tangential projection $\tau_{h-1}^X$ is not birational. Note that in general the degree of $\tau_{h-1}^X$ can be much less than $d_h$ as the following example shows. Let $X\subset\mathbb{P}^3$ be a general projection of a rational normal curve $C\subset\mathbb{P}^d$. A straightforward computation shows that a general tangent line of $X$ intersects $X$ only at the tangency point with multiplicity two. On the other hand, through a general point of $\mathbb{P}^3$ there are $\frac{(d-1)(d-2)}{2}$ lines that are secant to $X$. Finally, $d_2 = \frac{(d-1)(d-2)}{2} > d-2$ for $d > 3$.      
\end{Remark}

\begin{Remark}\label{rem:curv_no_bro}
Proposition \ref{Bro_cur} does not hold for reducible curves. For
instance, let $X = L_1\cup L_2\cup R_1\cup R_2\subset\mathbb{P}^3$ be
the union of four lines such that $L_i$ intersects $R_1,R_2$ in a
point for $i = 1,2$ and $L_1\cap L_2 = R_1\cap R_2 = \emptyset$. Let
$p\in \mathbb{P}^3$ be a general point. There are two lines passing through $p$ and intersecting $X$ in two points, one for the pair of lines $L_1,L_2$ and the other for the pair of
lines $R_1,R_2$. Hence the curve $X$ is not $2$-identifiable.

Now, take a point $x\in X$ which is a general point of one of its irreducible components say
$L_1$. Then the tangential projection $\tau_x^X$ is not defined on
$L_1$, contracts $R_1,R_2$ and maps $L_2$ birationally onto
$\mathbb{P}^1$. So $\tau_x^X:X\dasharrow\mathbb{P}^1$ is birational.

This should be confronted with the examples of varieties that are
$h$-tangentially weakly defective and $h$-identifiable. In all the known examples with this behavior the tangential contact locus has a geometry similar to that of the reducible curve $X$.
\end{Remark}

We are ready to build the announced counterexamples to the Bronowski's conjecture. We start with an elementary result which is probably classical though we did not find any reference.

\begin{Lemma}\label{lspg}
Let $\Lambda_1,\dots,\Lambda_h\subset\mathbb{P}^{hr-1}$ be
$(r-1)$-planes whose linear span is the whole ambient
projective space $\mathbb{P}^{hr-1}$, and let $p\in \mathbb{P}^{hr-1}$
be a general point. There is a unique $(h-1)$-plane $\Lambda^{h-1}$
such that $p\in\Lambda^{h-1}$ and $\Lambda^{h-1}$ intersects
$\Lambda_i$ in a point for all $i = 1,\dots,h$.   
\end{Lemma}
\begin{proof}
Consider the case $h = 2$. Then $\Lambda = \left\langle p,
  \Lambda_1\right\rangle$ is an $r$-plane and since
$\Lambda,\Lambda_2\subset\mathbb{P}^{2r-1}$ we have that
$\Lambda\cap\Lambda_2 = \{q\}$. Since the line $\left\langle p,
  q\right\rangle$ and $\Lambda_1$ are contained in $\Lambda$ they
intersect in a point and hence $\Lambda^{h-1} = \left\langle p,q\right\rangle$ is the line we are looking for. 

Now, consider the general case. Set $\Lambda = \left\langle p,\Lambda_1,\dots,\Lambda_{h-1}\right\rangle$. Then $\Lambda$ is an $(r(h-1))$-plane containing the $(r(h-1)-1)$-plane
$\left\langle\Lambda_1,\dots,\Lambda_{h-1}\right\rangle$. Hence
$\Lambda$ intersects $\Lambda_h$ in a point $q$, and the line
$\left\langle p,q\right\rangle$ intersects
$\left\langle\Lambda_1,\dots,\Lambda_{h-1}\right\rangle$ in a point
$p'$. By induction on $h$ there is an $(h-2)$-plane $\Lambda^{h-2}\subset
\left\langle\Lambda_1,\dots,\Lambda_{h-1}\right\rangle$ passing
through $p'$ and intersecting $\Lambda_1,\dots,\Lambda_{h-1}$. So the
$(h-1)$-plane $\Lambda^{h-1} = \left\langle q,\Lambda^{h-2}\right\rangle$ intersects
$\Lambda_1,\dots,\Lambda_{h}$ and since $\left\langle p,q\right\rangle\subset \Lambda^{h-1}$ it passes through $p$.  

Assume that there were two distinct $(h-1)$-planes $\Lambda^{h-1},
\Gamma^{h-1}$ through $p$ and intersecting $\Lambda_i$ in a point for all $i = 1,\dots,h$. Then $H = \left\langle \Lambda^{h-1}, \Gamma^{h-1}\right\rangle$ would be a linear subspace of dimension
at most $2(h-1)$ intersecting at least one of the $\Lambda_i$ along
a line. Hence $\Lambda_1,\dots,\Lambda_h$ could not span the whole ambient projective space $\mathbb{P}^{hr-1}$.  
\end{proof}

\begin{Proposition}\label{secnoId}
Let $Y\subset\mathbb{P}^N$ be an irreducible and non-degenerate variety and set $X = \sec_r(Y)$. Then $X$ is not $h$-identifiable for all $h\geq 2$.
\end{Proposition}
\begin{proof}
Let $p\in \sec_h(X)$ be a general point. Then $p\in\left\langle
  x_1,\dots,x_h\right\rangle$ where $x_i\in X$, $x_i\in\left\langle
  y_1^i,\dots,y_r^i\right\rangle$ for
$y_1^1,\dots,y_r^1,\dots,y_1^h,\dots,y_r^h \in Y$ general points. Let
$\Lambda$ be the span of the points $\{y_{j}^{i}\}$. Then $\Lambda$ has dimension
$hr-1$ and for any choice of $h$ disjoint subsets of $r$ points among the $y_{j}^{i}$ we
have linear subspace $\Lambda_1,\dots,\Lambda_h$ of dimension
$r-1$ spanning $\Lambda$. By Lemma \ref{lspg} for any such choice
there exists an $(h-1)$-plane $\Lambda^{h-1}$ through $p$ intersecting
$\Lambda_1,\dots,\Lambda_h$, and hence $X$ is not $h$-identifiable. 
\end{proof}

\begin{thm}\label{mainCE}
Let $Y\subset\mathbb{P}^N$ be an irreducible and non-degenerate variety. Set $X(Y,r) = \sec_r(Y)$ and let $\tau_{r(h-1)}^Y:Y\dasharrow Y_{r(h-1)}$ be a general $r(h-1)$-tangential projection of $Y$. Assume that
\begin{itemize}
\item[-] $X(Y,r)$ is not $h$-defective and $h\dim(X(Y,r))+h-1\leq N$; and
\item[-] $\tau_{r(h-1)}^Y:Y\dasharrow Y_{r(h-1)}$ is birational and $Y_{r(h-1)}\subset\mathbb{P}^{N_{r(h-1)}}$ is $r$-identifiable. 
\end{itemize} 
Then we have that
\begin{itemize}
\item[-] a general $(h-1)$-tangential projection $\tau_{h-1}^{X(Y,r)}:X(Y,r)\dasharrow X(Y,r)_{h-1}$
is birational; and
\item[-] $X(Y,r)$ is not $h$-identifiable.
\end{itemize}
In particular, $X(Y,r)$ is not an $h$-Bronowski's variety, and if $X(Y,r)_{h-1}\subset\mathbb{P}^{N_{r(h-1)}}$ is a variety of minimal degree then $X(Y,r)$ provides a counterexample to Conjecture \ref{CBro} while when $X(Y,r)_{h-1} = \p^{N_{r(h-1)}}$ it provides a counterexample to both Conjecture \ref{CBroC} and Conjecture \ref{CBro}.
\end{thm}
\begin{proof}
By the Terracini's lemma \cite{Te11} the tangential projection $\tau_{h-1}^{X(Y,r)}:X(Y,r)\dasharrow X(Y,r)_{h-1}$ restricts to the tangential projection $\tau_{r(h-1)}^Y:Y\dasharrow Y_{r(h-1)}$ and $X(Y,r)_{h-1} = \sec_{r}(Y_{r(h-1)})$.  

Assume that $\tau_{h-1}^{X(Y,r)}:X(Y,r)\dasharrow X(Y,r)_{h-1}$ is not birational. Then
for $p\in X(Y,r)_{h-1}$ general there are two points $x_1\in\left\langle
  y_1^1,\dots,y_r^1\right\rangle$, $x_2\in\left\langle
  y_1^2,\dots,y_r^2\right\rangle$ such that $\tau_{h-1}^{X(Y,r)}(x_1) =
\tau_{h-1}^{X(Y,r)}(x_2) = p$. Since
$Y_{r(h-1)}\subset\mathbb{P}^{N_{r(h-1)}}$ is $r$-identifiable we must
have 
$$\tau_{r(h-1)}^Y(\{y_1^1,\dots,y_r^1\}) =
\tau_{r(h-1)}^Y(\{y_1^2,\dots,y_r^2\}),$$ 
contradicting the birationality of $\tau_{r(h-1)}^Y:Y\dasharrow Y_{r(h-1)}$. Therefore,
$\tau_{h-1}^{X(Y,r)}:X(Y,r)\dasharrow X(Y,r)_{h-1}$ is birational and
hence Proposition \ref{secnoId} implies that $X(Y,r)$ is not an
$h$-Bronowski's variety.
\end{proof}
 
We are ready to provide the counterexamples announced in Theorem~\ref{mainA}.
\begin{proof}[Proof of Theorem~\ref{mainA}]
To conclude we prove that the secants of a  rational normal curve of degree $N\geq 7$ have the properties listed in Theorem~\ref{mainCE}. 

Let  $\Gamma\subset\mathbb{P}^N$ be a rational normal curve of degree $N\geq 7$. Set $X(\Gamma,r) = \sec_r(\Gamma)$ and consider $\sec_h(X(\Gamma,r))$ with $2\leq h\leq\lfloor\frac{N+1}{2}\rfloor$. 

Since $\sec_h(X(\Gamma,r)) = \sec_{hr}(\Gamma)$ and $\Gamma$ is not $hr$-defective we have that $X(\Gamma,r)$ is not $h$-defective.  

Furthermore, $\Gamma_{r(h-1)}\subset \mathbb{P}^{N_{r(h-1)}}$ is a rational normal curve of degree $N-2r(h-1)$, and hence $\tau_{r(h-1)}^\Gamma:\Gamma\dasharrow \Gamma_{r(h-1)}$ is birational and $\Gamma_{r(h-1)}\subset\mathbb{P}^{N_{r(h-1)}}$ is $r$-identifiable.  

So $X(\Gamma,r) = \sec_r(\Gamma)\subset\mathbb{P}^N$ is not an $h$-Bronowski's variety. Moreover, when $h = \frac{N+1}{2r}$ is an integer then $X(\Gamma,r)_{h-1} = \mathbb{P}^{N_{r(h-1)}}$ is of minimal degree and hence $X(\Gamma,r)$ provides a counterexample to both Conjecture \ref{CBroC} and Conjecture~\ref{CBro}. 
\end{proof}

\begin{Remark}
The argument in the proof of Theorem~\ref{mainA} does not work when $\Gamma\subset\mathbb{P}^N$ is not a rational normal curve. Assume $h = \frac{N+1}{2r}$ is an integer so that $N_{r(h-1)} = 2r-1$. If $\Gamma\subset\mathbb{P}^N$ is not a rational normal curve then also $\Gamma_{r(h-1)}\subset \mathbb{P}^{N_{r(h-1)}}$ is not a rational normal curve and so it is not $r$-identifiable. Indeed, assume that $\Gamma_{r(h-1)}\subset\mathbb{P}^{2r-1}$ is $r$-identifiable. Then \cite[Corollary 4.5]{CR06} yields that a general $(r-1)$-tangential projection of $\Gamma_{r(h-1)}$ is birational, so $\Gamma_{r(h-1)}$ is rational. On the other hand, since $\tau_{r-1}^{\Gamma_{r(h-1)}}$ is birational a general hyperplane containing the span of $r-1$ general tangent spaces of $\Gamma_{r(h-1)}$ intersects $\Gamma_{r(h-1)}$ in a $0$-dimensional scheme of degree $2r-1 = 2(r-1)+1$. Therefore, $\deg(\Gamma_{r(h-1)}) = 2r-1$ and $\Gamma_{r(h-1)}\subset\mathbb{P}^{2r-1}$ is a rational normal curve.  
    
For instance, let $\Gamma\subset\mathbb{P}^7$ be a general projection
of a rational normal curve of degree eight in $\mathbb{P}^8$, and
$X(\Gamma,2) = \sec_2(\Gamma)$. By the argument in the proof of
Theorem \ref{mainCE} we have that $X(\Gamma,2)$ is not
$2$-identifiable. Furthermore, $\tau_{2}^\Gamma:\Gamma\dasharrow
\Gamma_{2}$ is birational and $\Gamma_2\subset\mathbb{P}^3$ is a
rational curve of degree four. Hence, it is not $2$-identifiable:
through a general point $p\in \mathbb{P}^3$ there are  three lines secant to $\Gamma_2$. On the other hand, $\Gamma\subset\mathbb{P}^7$ is $2$-identifiable and so $\tau_{1}^{X(\Gamma,2)}:X(\Gamma,2)\dasharrow\mathbb{P}^3$ must have degree at least three.   

Furthermore, when $X(\Gamma,r)_{h-1} \neq \p^{N_{r(h-1)}}$ we do not expect $X(\Gamma,r)_{h-1} \subsetneq \p^{N_{r(h-1)}}$ to have minimal degree. For instance, already for the simplest case in which $\Gamma\subset\mathbb{P}^8$ is a rational normal curve of degree eight and $\Gamma_2\subset\mathbb{P}^4$ is a rational normal curve of degree four we have that $X(\Gamma,2)_{1} = \sec_2(\Gamma_2)\subset \mathbb{P}^4$ is a cubic hypersurface.  
\end{Remark}

\section{Identifiability and the Bronowski's conjecture}\label{sec3}

In this section we prove Theorem~\ref{mainB}. We start by recalling some results about tangential contact loci.

\begin{Lemma}\label{CLoc} 
Let $X\subset\mathbb{P}^N$ be an irreducible and non-degenerate variety. Let $A\subset X$ be a set of $h$ general points and $\Gamma_h$ the associated contact locus. Then we have that:
\begin{itemize}
\item[(a)] $\Gamma_h$ is equidimensional and it is either irreducible or reduced with exactly $h$ irreducible component, each of them containing a single point of $A$;
\item[(b)] $\Span{\Gamma_h}=\sec_h(\Gamma_h)$ and $\sec_i(\Gamma_h)\neq \Span{\Gamma_h}$ for $i<h$;
\item[(c)] if $\Gamma_h$ is irreducible then the $h$-secant degree of $X$ is equal to the $h$-secant degree of $\Gamma_h$. 
\end{itemize}
\end{Lemma}
\begin{proof}
Items (a) and (b) are in \cite[Proposition 3.9]{CC10}, and (c) is in \cite[Corollary 4.5]{BBC}. 
\end{proof}

Next, even if not strictly necessary, we take the opportunity to
strengthen \cite[Proposition 3,5]{CC02} with the following lemma.
\begin{Lemma}\label{taupi}
Let $X\subset\mathbb{P}^N$ be an irreducible and non-degenerate variety of dimension $n$. Assume that $X$ is not $h$-defective and $(h+1)n+h\leq N$. Then a general fiber of $\tau_h^X$ and a general fiber of $\pi_{h+1}^X$ have the same dimension. 
\end{Lemma}
\begin{proof}
Take $x_1,\dots,x_{h+1}\in X$ general points and set 
$$\delta = \dim(\left\langle T_{x_1}X,\dots,T_{x_{h}}X \right\rangle \cap T_{x_{h+1}}X).$$ 
By the Terracini's lemma \cite{Te11} the general fiber of $\pi_{h+1}^X$ has dimension
$$
(h+1)n+h - (hn+h-1+n-\delta) = \delta + 1.
$$ 
To conclude it is enough to note that a general tangent space of the image $\tau_h^X(X)$ of $X$ via the $h$-tangential projection has dimension $n-\delta-1$.
\end{proof}

The first step to prove Theorem~\ref{mainB} consists in studying the case of varieties with $\pi_h^X$ not birational and contact locus $\Gamma_h$ given by an irreducible curve.

\begin{Proposition}\label{cldim1}
Let $X\subset\mathbb{P}^N$ be an irreducible and non-degenerate variety. Assume that 
\begin{itemize}
\item[-] $X$ is not $h$-identifiable; and
\item[-] the general $h$-tangential contact locus of $X$ is an irreducible curve.
\end{itemize} 
Then $\pi_{h+1}^X$ is of fiber type. 
\end{Proposition}
\begin{proof}
We may assume that $(h+1)n+h\leq N$ and that $X$ is not $h$-defective. Let $x_1,\dots,x_h\in X$ be general points and $\Gamma_{x_1,\dots,x_h}$ the corresponding contact locus. 

By hypothesis $\Gamma_{x_1,\dots,x_h}$ is an irreducible curve. Choose
$y_1,\dots,y_{h-1}\in \Gamma_{x_1,\dots,x_h}$ and $x\in X$
general. The tangential projection $\tau_{x_1,\dots,x_h}^X$ restricts
on $\Gamma_{y_1,\dots,y_{h-1},x}$ to the tangential projection
$\tau_{y_1,\dots,y_{h-1}}^{\Gamma_{y_1,\dots,y_{h-1},x}}$. The contact
locus
$\Gamma_{y_1,\dots,y_{h-1},x}$ is an irreducible, not $h$-identifiable
curve and, by Lemma~\ref{CLoc} we have
$$\sec_{h}(\Gamma_{y_1,\dots,y_{h-1},x})=\Span{\Gamma_{y_1,\dots,y_{h-1},x}}.$$
Then by Proposition~\ref{Bro_cur} the contact locus $\Gamma_{y_1,\dots,y_{h-1},x}$ is an $h$-Bronowski's
variety and 
$\tau_{y_1,\dots,y_{h-1}}^{\Gamma_{y_1,\dots,y_{h-1},x}}$ is not
birational.   This shows that, for a general choice of $y_1,\dots,y_{h-1}$
in $\Gamma_{x_1,\dots,x_h}$ there is a point in
$\Gamma_{y_1,\dots,y_{h-1},x}\setminus\{x\}$ that is mapped to
$\tau_{x_1,\dots,x_h}^X(x)$. Therefore by letting $y_1,\dots,y_{h-1}$ in $\Gamma_{x_1,\dots,x_h}$ vary we produce a positive dimensional subvariety of the fiber $(\tau_{x_1,\dots,x_{h}}^{X})^{-1}((\tau_{x_1,\dots,x_{h}}^X)(x))$. Hence $\tau_{x_1,\dots,x_{h}}^{X}$ is of fiber type, and Lemma \ref{taupi} yields that $\pi_{h+1}^X$ is of fiber type as well.  
\end{proof}
\begin{Remark}
Note that the proof of Proposition \ref{cldim1} could work for non $1$-tangentially weakly defective varieties with irreducible $h$-tangential contact loci of arbitrary dimension if Conjecture~\ref{con:Bro_mio} was proved. 
\end{Remark}

We are now ready to prove Theorem~\ref{mainB}.

\begin{proof}[Proof of Theorem~\ref{mainB}] Note that under our hypotheses $\pi_{h+1}^X$ is a generically finite morphism and $X$ is not $1$-tangentially weakly defective.  

Assume, by contradiction, that $X$ is not $h$-identifiable and let $\Gamma_h$ be a general $h$-tangential contact locus. By Proposition~\ref{cldim1} we may assume that $\Gamma_h$ is not an irreducible curve.

If $X$ is $h$-defective then $\pi_{h}^X$ is of fiber type and hence $\pi_{h+1}^X$ is also of fiber type. Therefore, we may assume that $X$ is not $h$-defective that is $\Span{T_{x_1}X,\dots,T_{x_{h-1}}X}\cap T_{x_h}X = \emptyset$.

Let $\{x_1,\dots,x_h,y\}\subset X$ be a set of general points and $\Gamma_{h+1} = \Gamma_{x_1,\dots,x_h,y}$ the associated $(h+1)$-tangential contact locus. To conclude the proof it is enough to prove that $\Span{T_{x_1}X,\dots,T_{x_h}X}\cap T_pX\neq\emptyset$ for $p\in \Gamma_{h+1}$ a general point.
  
For a general point $p\in \Gamma_{h+1}$ we will denote by $\Gamma^1_p$ the irreducible component through $p$ of the $h$-tangential contact locus $\Gamma_{x_1,\dots,x_{h-1},p}$ and by $\Gamma^h_p$ the irreducible component through $p$ of the $h$-tangential contact locus $\Gamma_{x_2,\dots,x_{h},p}$, note that by \cite[Proposition 3.9]{CC10} such irreducible components are unique. 

By hypothesis a general $h$-tangential contact locus of $X$ is either of dimension greater that one or a reducible curve. Therefore, we can choose a curve $\widetilde{\Gamma}^i_p\subset\Gamma^i_p$ through $p$ and not containing $x_i$ for $i = 1,h$. Note that for a general point $x\in\widetilde{\Gamma}^1_p$ we have $T_xX\subset\Span{T_{x_1}X,\dots,T_{x_{h-1}}X,T_pX}$. Then by semicontinuity for any point $w\in\widetilde{\Gamma}^1_p$ there is an $n$-plane $A_w\subseteq T_wX$ contained in $\Span{T_{x_1}X,\dots,T_{x_{h-1}}X,T_pX}$ which is the flat limit of the tangent spaces of $X$ along $\widetilde{\Gamma}^1_p$. Set
$$T(\widetilde{\Gamma}^1_p)=\Span{A_w}_{w\in\widetilde{\Gamma}^1_p},$$
$\dim(T(\widetilde{\Gamma}^1_p)) = n+a$, and 
$$
Y(\widetilde{\Gamma}^1_p)= \overline{\bigcup_{x\in\widetilde{\Gamma}^1_p \text{ general}}T_xX}.
$$
Note that $Y(\widetilde{\Gamma}^1_p)\subset T(\widetilde{\Gamma}^1_p)$. Since $X$ is not $1$-tangentially weakly defective we have that $\dim(Y(\widetilde{\Gamma}^1_p)) = n+1$ and consequently $a \geq 1$. 

Furthermore, since 
$$Y(\widetilde{\Gamma}^1_p),T(\widetilde{\Gamma}^1_p),\Span{T_{x_1}X,\dots,T_{x_{h-1}}X}\subset \Span{T_{x_1}X,\dots,T_{x_{h-1}}X,T_pX}$$ 
we get
\stepcounter{thm}
\begin{equation}\label{equ0}
\dim(T(\widetilde{\Gamma}^1_p)\cap \Span{T_{x_1}X,\dots,T_{x_{h-1}}X}) = a-1\geq 0,
\end{equation}
and
\stepcounter{thm}
\begin{equation}\label{equ1}
\dim(Y(\widetilde{\Gamma}^1_p)\cap \Span{T_{x_1}X,\dots,T_{x_{h-1}}X}) \geq 0.
\end{equation}
In particular, (\ref{equ1}) yields that there is a $z\in\widetilde{\Gamma}^1_p$ such that 
$$A_z\cap \Span{T_{x_1}X,\dots,T_{x_{h-1}}X}\neq \emptyset,$$ 
and since $x_1\notin\widetilde{\Gamma}^1_p$ we have $z\neq x_1$. Moreover, since $x_1,\dots,x_{h-1}\in X$ are general we may assume that 
\stepcounter{thm}
\begin{equation}\label{equ*}
A_z\cap \Span{T_{x_2}X,\dots,T_{x_{h-1}}X} = \emptyset.
\end{equation}
First, assume that $A_z\cap T_{x_h}X\neq\emptyset$. Then (\ref{equ0}) implies that 
$$
\dim(T(\widetilde{\Gamma}^1_p)\cap \Span{T_{x_1}X,\dots,T_{x_{h}}X}) \geq a\geq 1.
$$
Set $B_p = T(\widetilde{\Gamma}^1_p)\cap
\Span{T_{x_1}X,\dots,T_{x_{h}}X}$. Then $\dim(B_p)\geq a$. Since
$T_pX$ and $B_p$ both live in $T(\widetilde{\Gamma}^1_p)$, which has
dimension $n+a$, we get that $\dim(B_p\cap T_pX)\geq 0$. Therefore,
$B_p\cap T_pX\neq\emptyset$ and since $B_p\subset
\Span{T_{x_1}X,\dots,T_{x_{h}}X}$ we conclude that
$\Span{T_{x_1}X,\dots,T_{x_{h}}X}\cap T_pX\neq\emptyset$ as well. A contradiction.

Now, assume that $A_z\cap T_{x_h}X = \emptyset$ and consider $\Span{A_z,T_{x_2}X,\dots,T_{x_{h}}X}$. By semicontinuity to this linear space we may associate an $h$-tangential contact locus $\Gamma_{x_2,\dots,x_h,z}$, and we consider its irreducible component $\Gamma_z^h$ passing through $z$. Set 
$$
T(\Gamma_z^h) = \left\langle A_w\right\rangle_{w\in\Gamma_z^h}
$$
and $\dim(T(\Gamma_z^h)) = n+a$. Arguing as in the first part of the proof we get that $\dim(T(\Gamma_z^h)\cap \Span{T_{x_2}X,\dots,T_{x_{h}}X}) = a-1\geq 0$, and hence (\ref{equ*}) yields that $\dim(T(\Gamma_z^h)\cap \Span{T_{x_1}X,\dots,T_{x_{h}}X}) = a\geq 1$. Set $B_z = T(\Gamma_z^h)\cap \Span{T_{x_1}X,\dots,T_{x_{h}}X}$. Then $\dim(B_z) = a \geq 1$ and for all $w\in \Gamma_z^h$ we have $A_w,B_z\subset T(\Gamma_z^h)$ and $\dim(A_w\cap B_z)\geq 0$. So 
\stepcounter{thm}
\begin{equation}\label{equ(1)}
A_w\cap \Span{T_{x_1}X,\dots,T_{x_{h}}X} \neq \emptyset \text{ for all } w\in\Gamma_z^h. 
\end{equation}
Moreover, since $z\neq x_1$, the points $x_1,\dots,x_h$ are general and $X$ is not $h$-defective we have that
\stepcounter{thm}
\begin{equation}\label{equ(2)}
A_w\cap \Span{T_{x_1}X,\dots,T_{x_{h-1}}X} = \emptyset \text{ for } w\in\Gamma_z^h \text{ general}. 
\end{equation}
Now, take a general point $w\in\Gamma_z^h$ and consider the
irreducible component $\Gamma_w^1$ of the $h$-tangential contact locus
$\Gamma_{x_1,\dots,x_{h-1},w}$ through $w$, note that by
\cite[Proposition 3.9]{CC10} such irreducible component is
unique. Since $z\neq x_1$ and $x_1,\dots,x_{h-1}$ are general we have
$z\notin\Gamma_w^1$ and so $\Gamma_w^1\neq\Gamma_z^h$.  The point
$z\in\Gamma^1_p$ is general, therefore  for
any  proper subvariety $W\subset
\Gamma_p^1$ we have 
$$
\Gamma_z^h\subset\overline{\bigcup_{v\in\Gamma^1_p\setminus W}\Gamma_v^h}.
$$
Set
$$Y=\overline{\bigcup_{v\in\Gamma^1_p \text{ general}}\Gamma_v^h}.$$
Then $\Gamma_z^h$ is contained in $Y$, and for $q\in\Gamma^1_p$ general we have that 
$$
\Gamma^1_q\subseteq \overline{\bigcup_{w\in\Gamma^h_p \text{ general}}\Gamma_w^1}.
$$
Hence $\Gamma_p^1$ is contained in
$$
Z = \overline{\bigcup_{w\in\Gamma^h_z \text{ general}}\Gamma_w^1}
$$
and a general point of $Z$ is a general point of $X$. Set
$$
T(\widetilde{\Gamma}_w^1) = \left\langle A_v\right\rangle_{v\in\Gamma_w^1}
$$
and $\dim(T(\widetilde{\Gamma}_w^1)) = n+a$. Arguing as in the first part of the proof we have that $\dim(T(\widetilde{\Gamma}_w^1)\cap \left\langle T_{x_1}X,\dots,T_{x_{h-1}}X\right\rangle) = a-1\geq 0$, and hence (\ref{equ(1)}) and (\ref{equ(2)}) yield $\dim(T(\widetilde{\Gamma}_w^1)\cap \left\langle T_{x_1}X,\dots,T_{x_{h}}X\right\rangle) \geq a$. Set $B_w = T(\widetilde{\Gamma}_w^1)\cap \left\langle T_{x_1}X,\dots,T_{x_{h}}X\right\rangle$. So $\dim(B_w)\geq a$. 

Let $\xi\in Z$ be a general point. Then $\xi\in\widetilde{\Gamma}^1_w$ for some $w\in\Gamma_z^h$. Hence, $A_{\xi},B_w\subset T({\Gamma}_w^1)$ yield $\dim(A_{\xi}\cap B_w)\geq 0$ and so $A_{\xi}\cap \left\langle T_{x_1}X,\dots,T_{x_{h}}X\right\rangle\neq \emptyset$. Note that since $\xi \in Z$ is general and a general point of $Z$ is a general point of $X$ we have proved that $T_{x}X\cap \left\langle T_{x_1}X,\dots,T_{x_{h}}X\right\rangle\neq \emptyset$ for $x\in X$ general. Therefore, $\tau_h^X$ has positive dimensional fibers and hence $X$ is $(h+1)$-defective contradicting the hypotheses and concluding the proof.
\end{proof}
  
\begin{Remark}\label{n1twdn}
We stress that, as the following example shows, Theorem~\ref{mainB} in general does not hold for $1$-tangentially weakly defective varieties. Let $Y\subset\mathbb{P}^{11}$ be
a rational normal curve, and set $X = \sec_2(Y)$. Note that, by the
Terracini's lemma \cite{Te11}, $X$ is $1$-tangentially weakly
defective. Since $\sec_3(X) = \sec_6(Y) = \mathbb{P}^{11}$ we have
that $X$ is not $3$-defective. On the other hand, Proposition
\ref{secnoId} yields that $X$ is not $2$-identifiable. 
\end{Remark}

We conclude this section by putting Corollary~\ref{cor:Bro_mio_vera} in the
perspective of \cite[Theorem 1.5]{CK22} stating that
Conjecture~\ref{CBro} holds under the assumption that $\pi_{h+1}^X$ is
generically finite.

\begin{Remark}\label{rem:Th15} Let $X$ be an irreducible and non-degenerate variety with $\pi_{h+1}$ generically finite.  Assume that $X$ is not $1$-tangentially weakly defective. Then, by Theorem \ref{mainB}, $X$ is
$h$-identifiable and hence $\tau_{h-1}^X$ is birational.
This neither requires nor gives any information about the degree of
the involved varieties. 

Let us conclude by proving that there does not exist any $1$-tangentially weakly defective variety $X$ with $X_{h-1}$ of minimal degree. Assume that $X$ is $1$-tangentially weakly defective, $\pi_{h+1}^X$ is
finite and $X_{h-1}$ is of minimal degree. Fix $x_1,\dots,x_{h-1},x\in
X$ general points. Since $\pi_{h+1}^X$ is finite we have that
$(h+1)n+h\leq N$ and hence $2n+1\leq N_{h-1}$. The tangent space
$T_xX$ is tangent to $X$ along a positive dimensional subvariety
$Y\subset X$ and so $T_{\overline{x}}X_{h-1}$ is tangent to $X_{h-1}$
along $\overline{Y} = \overline{\tau_{x_1,\dots,x_{h-1}^X}(Y)}$, where
$\overline{x} = \tau_{x_1,\dots,x_{h-1}^X}(x)$.  

Since $\pi_{h+1}^X$ is finite $\tau_{x_1,\dots,x_{h-1}}^X$ is also
finite, and since $x_1,\dots,x_{h-1},x\in X$ are general
$\overline{Y}$ has positive dimension. Therefore, $X_{h-1}$ is a
$1$-tangentially weakly defective variety of minimal degree and by the
classification of the varieties of minimal degree \cite[Theorem
1]{EH87} $X_{h-1}$ must be a cone. In particular, $X_{h-1}$ is
$2$-defective that is $\tau^{X_{h-1}}_{\overline{x}}$ is of fiber
type. So $\tau_{x_1,\dots,x_{h-1},x}^X =
\tau_{\overline{x}}^{X_{h-1}}\circ\tau_{x_1,\dots,x_{h-1}}^{X}$ is
also of fiber type and hence $X$ is $(h+1)$-defective contradicting
the hypothesis on the finiteness of $\pi_{h+1}^X$.       
\end{Remark}

\section{Applications to toric varieties and varieties in tensor spaces}\label{sec4}

In this last section we apply Theorem \ref{mainB} to get identifiability results for many interesting classes of varieties expanding and completing \cite[Section 3]{CM19}. Most of the varieties that we will take into account are smooth and hence not $1$-tangentially weakly defective \cite[Remark 3.18]{Za93}.

\stepcounter{thm}
\subsection{Toric and Segre-Veronese varieties}
Let $N$ be a rank $n$ free abelian group, $M := {\rm Hom}(N,\mathbb Z)$ its dual and $M_{\mathbb Q} := M\otimes_{\mathbb Z}\mathbb Q$ the corresponding rational vector space. Let $P\subseteq M_{\mathbb Q}$ be a full-dimensional lattice polytope, that is the convex hull of finitely many points in $M$ which do not lie on a hyperplane. The polytope $P$ defines a polarized pair $(X_P,H)$ consisting of the toric variety $X_P$ together with a very ample Cartier
divisor $H$ of $X_P$. More precisely $X_P$ is the Zariski closure of the image of
the monomial map
$$
\begin{array}{lccc}
\phi_P: & (k^*)^n & \longrightarrow & \mathbb{P}^{N}\\
& u & \mapsto & [\chi^m(u)\, :\, m\in P\cap M]
\end{array}
$$
where $P\cap M = \{m_0,\dots,m_N\}$, $\chi^m(u)$ denotes the Laurent monomial in the variables $(u_1,\dots,u_n)$ defined by the point $m$, and $H$ is a hyperplane section of $X_P$. Given a subset $\{m_{i_0},\dots,m_{i_n}\}\subset P\cap M$ we will denote by $\Aff(m_{i_0},\dots,m_{i_n})$ the set of linear combinations $\sum_{j=0}^{n}a_jm_{i_j}$ with $a_j\in k$ and $\sum_{j}a_j = 1$. Set
$$
B = \{m_{i_0} +\dots + m_{i_n}\: | \: \Aff(m_{i_0},\dots,m_{i_n}) = M\otimes_{\mathbb Z} k\}
$$
and let $\left\langle B-B\right\rangle_{\mathbb{Q}}$ be the vector subspace of $M_{\mathbb{Q}}$ generated by $B-B = \{b-b'\: | \: b,b'\in B\}$. Finally, set $M' = \frac{M}{\left\langle B-B\right\rangle_{\mathbb{Q}}\cap M}$.

When the polytope $P\subset M_{\mathbb{Q}}$ is a product of simplexes we get a Segre-Veronese variety. For these varieties we will follow the notation established in the introduction.

\begin{thm}\label{toric}
Let $P\subseteq M_{\mathbb{Q}}$ be a full-dimensional lattice polytope, $X_P\subseteq\mathbb P^{|P\cap M|-1}$ the corresponding $n$-dimensional toric variety, and $m$ the maximum number of integer points in a hyperplane section of $P$. If $\rk(M') = n$ and
$$h < \left\lfloor\dfrac{|P\cap M|-m}{n+1}\right\rfloor$$
then $X_P$ is $h$-identifiable. In particular, if $SV_{\pmb{d}}^{\pmb{n}} \subset \mathbb{P}^{N(\pmb{n},\pmb{d})}$ is a Segre-Veronese variety and
$$h < \left\lfloor\frac{d_j}{n_j+d_j}\frac{1}{\sum_{i=1}^r n_i + 1}\prod_{i=1}^r \binom{n_i+d_i}{d_i}\right\rfloor,$$
where $\frac{n_j}{d_j} = \max_{1\leq i\leq r}\left\lbrace\frac{n_i}{d_i}\right\rbrace$, then $SV_{\pmb{d}}^{\pmb{n}}$ is $h$-identifiable.
\end{thm}
\begin{proof}
By \cite[Theorem 2.13]{LMR20} $X_P$ is not $h$-defective for $h\leq\dfrac{|P\cap M|-m}{n+1}$. Furthermore, by \cite[Theorem 1.1]{FI17} $X_P$ is $1$-tangentially weakly defective if and only if $\rk(M') < n$. Hence, to conclude it is enough to apply Theorem \ref{mainB}. The claim on $SV_{\pmb{d}}^{\pmb{n}}$ follows from \cite[Theorem 1.1]{LMR20} and Theorem \ref{mainB}.
\end{proof}

\begin{proof}[Proof of Theorem~\ref{SVP1}]
By \cite[Theorem 3.1]{LP13} $SV_{\pmb{d}}^{(1,\dots,1)} \subset \mathbb{P}^{N(\pmb{n},\pmb{d})}$ is not $h$-defective except when 
$$(d_1,\dots,d_r;h) \in \{(2,2a;2a+1),(1,1,2a;2a+1),(2,2,2;7),(1,1,1,1;3)\}$$
with $a\geq 1$. Hence, thanks to the numerical hypothesis, \cite[Theorem 3.1]{LP13} and Theorem
\ref{mainB} it is enough to study the identifiability in the  following cases:
$$(d_1,\dots,d_r;h) \in \{(2,2a;2a),(1,1,2a;2a),(1,1,1,1;2)\}$$
for $a\geq 1$. By the proof of \cite[Proposition 7.1]{LP13} the linear system of divisors of bidegree $(2,2a)$ on $\mathbb{P}^1\times\mathbb{P}^1$ with $2a+1$ general double points has a unique section which is a smooth rational curve. Hence, the linear system of divisors of bidegree $(2,2a)$ in $\mathbb{P}^1\times\mathbb{P}^1$ with $2a$ general double points does not have any base component. In particular, $SV_{(2,2a)}^{(1,1)}$ is not $2a$-tangentially weakly defective and hence it is $2a$-identifiable. 

Again by the proof of \cite[Proposition 7.1]{LP13} the linear system of divisors of multidegree $(1,1,2a)$ on $\mathbb{P}^1\times\mathbb{P}^1\times\mathbb{P}^1$ with $2a+1$ general double points has a unique section which is the union of a section of the linear system of divisors of multidegree $(1,0,a)$ passing through the $2a+1$ points and a section of the linear system of divisors of multidegree $(0,1,a)$ passing through the $2a+1$ points. Therefore, the contact locus of the span of $2a$ general tangent space of $SV_{(1,1,2a)}^{(1,1,1)}$ is $0$-dimensional. So $SV_{(1,1,2a)}^{(1,1,1)}$ is not $2a$-tangentially weakly defective and hence it is $2a$-identifiable. 

Finally, consider $SV_{(1,1,1,1)}^{(1,1,1,1)}$. In this case a direct computation shows that the span of two general tangent spaces intersects $SV_{(1,1,1,1)}^{(1,1,1,1)}$ along eight lines, four of them through each point, and is tangent to $SV_{(1,1,1,1)}^{(1,1,1,1)}$ just at the two chosen points. So $SV_{(1,1,1,1)}^{(1,1,1,1)}$ is not $2$-tangentially weakly defective and hence it is $2$-identifiable.

To conclude observe that when $d_i\geq 3$ the variety $SV_{\bf d}^{(1,\ldots,1)}$ is never generically identifiable by \cite[Theorem 30]{CM21}.
\end{proof}

\begin{Remark}\label{SVn6id}
Consider now the case $(d_1,\dots,d_r;h) = (2,2,2;6)$. A direct computation shows that the span of six general tangent spaces is tangent to $SV_{(2,2,2)}^{(1,1,1)}$ along a degree twelve elliptic normal curve $\Gamma_6\subset\mathbb{P}^{11}$. By Lemma \ref{CLoc} (c) the $6$-secant degree of $SV_{(2,2,2)}^{(1,1,1)}$ is equal to that of $\Gamma_6$ which by \cite[Proposition 5.2]{CC06} is two. In particular, $SV_{(2,2,2)}^{(1,1,1)}$ is not $2$-identifiable. 
\end{Remark}

\stepcounter{thm}
\subsection{Grassmannians and Flag varieties}
Fix a vector space $V$ of dimension $n+1$ and integers $k_1\leq\dots\leq k_r$. Let $\mathbb{G}(k_i,n)\subset\mathbb{P}^{N_i}$, where $N_i={{n+1}\choose{k_i+1}}-1$, be the Grassmannian of $k_i$-dimensional linear subspaces of $\mathbb{P}(V)$ in its Pl\"ucker embedding. We have an embedding of the product of these Grassmannians 
$$\mathbb{G}(k_1,n)\times\dots\times\mathbb{G}(k_r,n)\subset\mathbb{P}^{N_1}\times\dots\times\mathbb{P}^{N_r}\subset\mathbb{P}^N$$
where $N={{n+1}\choose{k_1+1}}\cdots{{n+1}\choose{k_r+1}}-1$.

The flag variety $\F(k_1,\dots,k_r;n)$ is the set of flags, that is nested subspaces,
$V_{k_1}\subset\cdots\subset V_{k_r}\subsetneq V$. This is a subvariety of the product of Grassmannians $\prod_{i=1}^r\G(k_i,n)$. Hence, via a product of Pl\"ucker embeddings followed by a Segre embedding we get and embedding
$$\F(k_1,\dots,k_r;n)\hookrightarrow \mathbb{P}^{N_1}\times\dots\times\mathbb{P}^{N_r}\hookrightarrow\mathbb{P}^N.$$

Consider natural numbers $a_1,\dots,a_{n}$ such that $a_{k_1+1}=\cdots=a_{k_r+1}=1$ and $a_i=0$ for all $i\notin\{k_1+1,\dots,k_r+1\}$. Then, $\F(k_1,\dots,k_r;n)$ generates the subspace 
$$\mathbb{P}(\Gamma_{a_1,\dots,a_{n}})\subseteq \mathbb{P}\left(\bigwedge^{k_1+1}V\otimes\cdots\otimes\bigwedge^{k_r+1}V\right)\subseteq \mathbb{P}^N$$
where $\Gamma_{a_1,\dots,a_{n}}$ is the irreducible representation of $\mathfrak{sl}_{n+1} \C$ with highest weight $(a_1+\cdots+a_{n})L_1+\cdots+a_{n}L_{n}$, and $L_1+\dots + L_k$ is the highest weight of the irreducible representation $\bigwedge^{k}V$. We will denote $\Gamma_{a_1,\dots,a_{n}}$ simply by $\Gamma_a$. By the Weyl character formula we have that 
$$\dim \mathbb{P}(\Gamma_{a}) = \prod_{1\leq i < j\leq n+1}\frac{(a_i+\dots+a_{j-1})+j-i}{j-i}-1.$$
Furthermore, $\dim \F(k_1,\dots,k_r;n) = (k_1+1)(n-k_1)+\sum_{j=2}^r(n-k_j)(k_j-k_{j-1})$ and $\F(k_1,\dots,k_r;n)=\mathbb{P}(\Gamma_a)\cap \prod_{i=1}^r\G(k_i,n)\subset \mathbb{P}^N$.

Since Grassmannians of lines are always defective when dealing with $\mathbb{G}(r,n)$ we will assume that $r\geq 2$. 

\begin{thm}\label{flag_G}
Consider a flag variety $\F(k_1,\dots,k_r;n)$. Assume that $n\geq 2k_{j}+1$ for some index $j$ and let $l$ be the maximum among these $j$. Then, for 
$$h < \left\lfloor\left(\frac{n+1}{k_l+1}\right)^{\lfloor \log_2(\sum_{j=1}^l k_j+l-1)\rfloor}\right\rfloor$$
$\F(k_1,\dots,k_r;n)$ is $h$-identifiable. In particular, for 
$$
h < \left\lfloor\left(\frac{n+1}{r+1}\right)^{\lfloor\log_2(r)\rfloor}\right\rfloor
$$
the Grassmannian $\mathbb{G}(r,n)\subset\mathbb{P}^N$, in its Pl\"ucker embedding, is $h$-identifiable. Furthermore, if $r\geq 2$, $h\leq 12$ and
$$
h < \left\lfloor \frac{\binom{n+1}{r+1}}{(n-r)(r+1)+1}\right\rfloor
$$
then $\mathbb{G}(r,n)\subset\mathbb{P}^N$ is $h$-identifiable except for the case $(r,n;h) = (3,7;3)$ in which the corresponding Grassmannian is $3$-defective. 
\end{thm}
\begin{proof}
The first claim follows from the main theorem in \cite{MR19}, \cite[Theorem 1.1]{BCM22} and Theorem \ref{mainB}. By \cite[Theorem 1.1]{Bo13} if $r\geq 2$ and $h\leq 12$ the Grassmannian $\mathbb{G}(r,n)\subset\mathbb{P}^N$ is not $h$-defective with the following exceptions:
$$
(r,n;h) \in \{(2,6;3),(3,7;3),(3,7;4),(2,8;4)\}.
$$
Therefore, thanks to Theorem \ref{mainB} it is enough to check the cases $(r,n;h) \in \{(3,7;2),(2,8;3)\}$ whose identifiability follows from \cite[Theorem 1.1]{BV18}.
\end{proof}

For Grassmannians of planes we provide a better bound. 

\begin{Proposition}\label{Gr2}
Consider the Grassmannian $\mathbb{G}(2,n)\subset\mathbb{P}^N$, in its Pl\"ucker embedding. If $n\geq 9$ and
$$
h < \left\lfloor \frac{n^2}{18}-\frac{20n}{27}+\frac{287}{81}\right\rfloor + \left\lfloor \frac{6n-13}{9}\right\rfloor
$$
then $\mathbb{G}(2,n)$ is $h$-identifiable. 
\end{Proposition}
\begin{proof}
It is enough to apply \cite[Theorem 1.5]{AOP12} and Theorem \ref{mainB}.
\end{proof}

\stepcounter{thm}
\subsection{Lagrangian Grassmannians and Spinor varieties}
Let $V$ be a vector space endowed with a non-degenerate quadratic form $Q$ or, when $\dim(V)$ is even, with a non-degenerated symplectic form $\omega$. For $r\leq\frac{\dim(V)}{2}$ the isotropic Grassmannians $\mathcal{G}_{Q}(r,V), \mathcal{G}_{\omega}(r,V)$ are the subvarieties of the Grassmannian $\mathcal{G}(r,V)$ parametrizing $r$-dimensional subspaces of $V$ that are isotropic with respect to $Q$ and $\omega$ respectively.

All isotropic Grassmannians, with the exception of the symmetric case when $\dim(V) = 2n$ is even and $r = n$, are irreducible. In the exceptional case the isotropic Grassmannian $\mathcal{G}_{Q}(n,V)$ has two connected components each one parametrizing the linear subspaces in one of the two families of $(n-1)$-planes of $\mathbb{P}(V)$ contained in the smooth quadric hypersurface in $\mathbb{P}(V)$ defined by $Q$. Either of these two isomorphic components is called the $\frac{n(n-1)}{2}$-dimensional Spinor variety and denoted by $\mathcal{S}_n$.

The restriction of the Pl\"ucker embedding of $\mathcal{G}(n,V)$ induces an embedding $\mathcal{S}_n\rightarrow\mathbb{P}(\bigwedge^nV_{+})$. However, this is not the minimal homogeneous embedding of $\mathcal{S}_n$ that we will denote by $\mathcal{S}_n\rightarrow\mathbb{P}(\Delta)$ and refer to as the Spinor embedding. The Pl\"ucker embedding of $\mathcal{S}_n$ can be obtained by composing the Spinor embedding with the degree two Veronese embedding. 

In the skew-symmetric case, again when $d = 2n$ is even and $r = n$, the isotropic Grassmannian $\mathcal{G}_{\omega}(n,V)$ is called the $\frac{n(n+1)}{2}$-dimensional Lagrangian Grassmannian and denoted by $\mathcal{LG}(n,2n)$. Unlike the case of the Spinor variety, the restricting of the Pl\"ucker embedding of $\mathcal{G}(n,V)$ yields the minimal homogeneous embedding of $\mathcal{LG}(n,2n)$ that we will denote by $\mathcal{LG}(n,2n)\rightarrow\mathbb{P}(V_{\omega_n})$.

\begin{thm}\label{LGS}
The Lagrangian Grassmannian $\mathcal{LG}(n,2n)\subset\mathbb{P}(V_{\omega})$, in its Pl\"ucker embedding, is $h$-identifiable for $h < \left\lfloor\frac{n+1}{2}\right\rfloor$. 

Furthermore, the Spinor variety $\mathcal{S}_n\subset\mathbb{P}(\bigwedge^nV_{+})$, in its Pl\"ucker embedding, is $h$-identifiable for $h<\left\lfloor\frac{n}{2}\right\rfloor$. Finally, the Spinor variety $\mathcal{S}_n\subset\mathbb{P}(\Delta)$, in its Spinor embedding, is $h$-identifiable for $h < \left\lfloor\frac{n+2}{4}\right\rfloor$.
\end{thm}
\begin{proof}
It is enough to apply \cite[Theorem 1.1]{FMR20} and Theorem \ref{mainB}.
\end{proof}

For small values of $h$ a little improvement is at hand. 

\begin{Proposition}\label{LGS2}
If $n\geq 5$ then the Lagrangian Grassmannian $\mathcal{LG}(n,2n)\subset\mathbb{P}(V_{\omega})$, in its Pl\"ucker embedding, is $h$-identifiable for $h \leq 3$. 

Furthermore, the Spinor variety $\mathcal{S}_n\subset\mathbb{P}(\Delta)$, in its Spinor embedding, is $2$-identifiable for all $n\notin\{7,8\}$.
\end{Proposition}
\begin{proof}
The claim about $\mathcal{LG}(n,2n)$ is a consequence of \cite[Theorem 1.1]{BB11} and Theorem \ref{mainB}, and the claim on $\mathcal{S}_n$ follows from \cite[Theorem 1.1]{An11} and Theorem \ref{mainB}.
\end{proof}

\stepcounter{thm}
\subsection{Gaussian moment varieties}
The Gaussian moment variety $\mathcal{G}_{1,d}\subset\mathbb{P}^d$ parametrizes the vectors of all moments of degree at most $d$ of a $1$-dimensional Gaussian distribution. By \cite[Corollary 2]{AFS16} $\mathcal{G}_{1,d}$ is a surface for all $d$.

\begin{thm}\label{gauss}
If $h < \lfloor\frac{d+1}{3}\rfloor$ then the Gaussian moment surface $\mathcal{G}_{1,d}$ is $h$-identifiable. 
\end{thm}
\begin{proof}
By \cite[Theorem 1]{ARS18} the surface $\mathcal{G}_{1,d}$ is not $h$-defective for all $h$, and by \cite[Example 5.8]{BBC} it is not $1$-tangentially weakly defective. As usual to conclude we apply Theorem \ref{mainB}.
\end{proof}

\stepcounter{thm}
\subsection{Varieties of powers of forms}
Let $\mathcal{V}_{a,d,n}\subset\mathbb{P}(k[x_0,\dots,x_n]_{ad})$ be the variety parametrizing $d$-powers of homogeneous polynomials of degree $a$. The defectiveness and identifiability of $\mathcal{V}_{a,d,n}$ have recently been studied in \cite{BCOM22}. Thanks to Theorem \ref{mainB} we get the following improvement on \cite[Corollary 4.4]{BCOM22}:

\begin{Proposition}
If $h \leq \frac{\binom{ad+n}{n}}{\binom{a+n}{n}}-\binom{a+n}{n}-1$ then $\mathcal{V}_{a,d,n}$ is $h$-identifiable. 
\end{Proposition}
\begin{proof}
By \cite[Proposition 4.1]{BCOM22} the variety $\mathcal{V}_{a,d,n}$ is smooth and hence it is not $1$-tangentially weakly defective. Hence, the claim follows from \cite[Theorem 4.3]{BCOM22} and Theorem \ref{mainB}.
\end{proof}

\begin{Remark}
Similar identifiability bounds for Chow-Veronese varieties and more generally for varieties parametrizing products of polynomials can be obtained from Theorem \ref{mainB} together with the non defectiveness results in \cite{CGGHMNS19} and \cite{TV21}.
\end{Remark}

\bibliographystyle{amsalpha}
\bibliography{Biblio} 
\end{document}